\documentclass[twocolumn,amsthm]{autart}
\usepackage{amssymb,amsmath,mathptmx}

\newcommand{\supC}{\mbox{$\sup {\rm C}$}}
\newcommand{\supCC}{\mbox{$\sup {\rm cC}$}}
\newcommand{\supcCSRO}{\mbox{$\sup {\rm cCSRO}$}}
\newcommand{\supCRO}{\mbox{$\sup {\rm CRO}$}}

\begin{document}

\begin{frontmatter}
\runtitle{A Note on Relative Observability in Coordination Control}
\title{A Note on Relative Observability in Coordination Control}

\thanks{Corresponding author: T. Masopust, TU Dresden, Germany}

\author{Jan Komenda}\ead{komenda{@}ipm.cz} 
\address{Institute of Mathematics, Academy of Sciences of the Czech Republic, {\v Z}i{\v z}kova 22, 616 62 Brno, Czech Republic}

\author{Tom{\' a}{\v s}~Masopust}\ead{masopust{@}math.cas.cz}
\address{Institute of Mathematics, Academy of Sciences of the Czech Republic, {\v Z}i{\v z}kova 22, 616 62 Brno, Czech Republic \and TU Dresden, Germany}

\author{Jan H. van Schuppen}\ead{jan.h.van.schuppen@xs4all.nl}
\address{Van Schuppen Control Research, Gouden Leeuw 143, 1103 KB, Amsterdam, The Netherlands}


\begin{abstract}
  Relative observability has been introduced and studied in the framework of partially observed discrete-event systems as a condition stronger than observability, but weaker than normality. However, unlike observability, relative observability is closed under language unions, which makes it interesting for practical applications. In this paper, we investigate this notion in the framework of coordination control. We prove that conditional normality is a stronger condition than conditional (strong) relative observability, hence conditional strong relative observability can be used in coordination control instead of conditional normality, and present a distributive procedure for the computation of a conditionally controllable and conditionally observable sublanguage of the specification that contains the supremal conditionally strong relative observable sublanguage.
\end{abstract} 

\end{frontmatter}

\section{Introduction}\label{intro}
  Supervisory control theory of discrete-event systems has been proposed in~\cite{RW89} as a formal approach to solve the safety issue and nonblockingness. Coordination control has been proposed for modular discrete-event systems in~\cite{KvS08} as a reasonable trade-off between a purely modular control synthesis, which is in some cases unrealistic, and a global control synthesis, which is naturally prohibitive for complexity reasons. The idea is to compute a coordinator that takes care of the communication between subsystems. This approach has been further developed in~\cite{automatica2011,JDEDS,cdc2014}. In~\cite{automatica2011}, a procedure for the distributive computation of the supremal conditionally-controllable sublanguages (the necessary and sufficient condition for the existence of a solution) of prefix-closed specification languages and controllers with complete observations has been proposed. The approach has been later extended to non-prefix-closed specification languages in~\cite{JDEDS}, and for partial observations in~\cite{cdc2014}.

  Relative observability has been introduced and studied in~\cite{caiCDC13} in the framework of partially observed discrete-event systems as a condition stronger than observability, but weaker than normality. Relative observability has been shown to be closed under language unions, which makes it an interesting notion that can replace normality in practical applications. Before relative observability, normality was the weakest notion known to be closed under language unions. 
  
  In this paper, we study the concept of relative observability in the coordination control framework. In the same manner as we have introduced the notions of conditional normality and conditional observability, we introduce and discuss the new concept of {\em conditional relative observability\/} in the coordination control framework. Surprisingly, compared to relative observability, conditional relative observability is not closed under language unions meaning that the supremal conditionally relative observable sublanguages do not always exist. Therefore, we further propose a stronger concept called {\em conditional strong relative observability}, which we show to be closed under language unions. Moreover, we prove that the previously defined notion of conditional normality~\cite{cdc2014} implies conditional (strong) relative observability, which means that conditional strong relative observability can be used in coordination control with partial observations instead of conditional normality, and we present a distributive procedure for the computation of a conditionally controllable and conditionally observable sublanguage of the specification that contains the supremal conditionally strong relative observable sublanguage.

\section{Preliminaries}\label{preliminaries}
  We first briefly recall the basic elements of supervisory control theory. The reader is referred to~\cite{CL08} for more details. Let $\Sigma$ be a finite nonempty set of {\em events}, and let $\Sigma^*$ denote the set of all finite words over $\Sigma$. The {\em empty word\/} is denoted by $\varepsilon$. 

  A {\em generator\/} is a quintuple $G=(Q,\Sigma, f, q_0, Q_m)$, where $Q$ is a finite nonempty set of {\em states}, $\Sigma$ is an {\em event set\/}, $f: Q \times \Sigma \to Q$ is a {\em partial transition function}, $q_0 \in Q$ is the {\em initial state}, and $Q_m\subseteq Q$ is the set of {\em marked states}. In the usual way, the transition function $f$ can be extended to the domain $Q \times \Sigma^*$ by induction. The behavior of $G$ is described in terms of languages. The language {\em generated\/} by $G$ is the set $L(G) = \{s\in \Sigma^* \mid f(q_0,s)\in Q\}$ and the language {\em marked\/} by $G$ is the set $L_m(G) = \{s\in \Sigma^* \mid f(q_0,s)\in Q_m\}\subseteq L(G)$.

  A {\em (regular) language\/} $L$ over an event set $\Sigma$ is a set $L\subseteq \Sigma^*$ such that there exists a generator $G$ with $L_m(G)=L$. The prefix closure of a language $L$ is the set $\overline{L}=\{w\in \Sigma^* \mid \text{there exists } u \in\Sigma^* \text{ such that } wu\in L\}$. A language $L$ is {\em prefix-closed\/} if $L=\overline{L}$.

  A {\em (natural) projection} $P: \Sigma^* \to \Sigma_o^*$, for some $\Sigma_o\subseteq \Sigma$, is a homomorphism defined so that $P(a)=\varepsilon$, for $a\in \Sigma\setminus \Sigma_o$, and $P(a)=a$, for $a\in \Sigma_o$. The {\em inverse image} of $P$, denoted by $P^{-1} : \Sigma_o^* \to 2^{\Sigma^*}$, is defined as $P^{-1}(s)=\{w\in \Sigma^* \mid P(w) = s\}$. The definitions can naturally be extended to languages. The {\em projection of a generator\/} $G$ is a generator $P(G)$ whose behavior satisfies $L(P(G))=P(L(G))$ and $L_m(P(G))=P(L_m(G))$.

  A {\em controlled generator\/} is a structure $(G,\Sigma_c,P,\Gamma)$, where $G$ is a generator over $\Sigma$, $\Sigma_c \subseteq \Sigma$ is the set of {\em controllable events}, $\Sigma_{u} = \Sigma \setminus \Sigma_c$ is the set of {\em uncontrollable events}, $P:\Sigma^*\to \Sigma_o^*$ is the projection, and $\Gamma = \{\gamma \subseteq \Sigma \mid \Sigma_{u}\subseteq\gamma\}$ is the {\em set of control patterns}. A {\em supervisor\/} for the controlled generator $(G,\Sigma_c,P,\Gamma)$ is a map $S:P(L(G)) \to \Gamma$. A {\em closed-loop system\/} associated with the controlled generator $(G,\Sigma_c,P,\Gamma)$ and the supervisor $S$ is defined as the minimal language $L(S/G) \subseteq \Sigma^*$ such that (i) $\varepsilon \in L(S/G)$ and (ii) if $s \in L(S/G)$, $sa\in L(G)$, and $a \in S(P(s))$, then $sa \in L(S/G)$. The marked behavior of the closed-loop system is defined as $L_m(S/G)=L(S/G)\cap L_m(G)$.
  
  Let $G$ be a generator over an event set $\Sigma$, and let $K\subseteq L_m(G)$ be a specification. The aim of supervisory control theory is to find a nonblocking supervisor $S$ such that $L_m(S/G)=K$; the nonblockingness means that $\overline{L_m(S/G)} = L(S/G)$, hence $L(S/G)=\overline{K}$. It is known that such a supervisor exists if and only if $K$ is 
  (i) {\em controllable\/} with respect to $L(G)$ and $\Sigma_u$, that is $\overline{K}\Sigma_u\cap L(G)\subseteq \overline{K}$, 
  (ii) {\em $L_m(G)$-closed}, that is $K = \overline{K}\cap L_m(G)$, and 
  (iii) {\em observable\/} with respect to $L(G)$, $\Sigma_o$, and $\Sigma_c$, that is for all words $s, s'\in \Sigma^*$ such that $Q(s) = Q(s')$ it holds that for all $\sigma\in \Sigma$, $s\sigma \in \overline{K}$, $s' \in \overline{K}$, and $s'\sigma \in L(G)$ imply that $s'\sigma \in \overline{K}$, where $Q:\Sigma^*\to \Sigma_o^*$. Note that it is sufficient to consider $\sigma\in\Sigma_c$, because for $\sigma\in\Sigma_u$ the condition follows from controllability, cf.~\cite{CL08}.
  
  The synchronous product of two languages $L_1\subseteq \Sigma_1^*$ and $L_2\subseteq \Sigma_2^*$ is defined by $L_1\parallel L_2=P_1^{-1}(L_1) \cap P_2^{-1}(L_2) \subseteq \Sigma^*$, where $P_i: \Sigma^*\to \Sigma_i^*$, for $i=1,2$, are projections to local event sets. In terms of generators, it is known that $L(G_1 \parallel G_2) = L(G_1) \parallel L(G_2)$ and $L_m(G_1 \parallel G_2)= L_m(G_1) \parallel L_m(G_2)$, see~\cite{CL08}.

\section{Coordination Control Framework}\label{sec:concepts}
  A language $K\subseteq (\Sigma_1\cup \Sigma_2)^*$ is {\em conditionally decomposable\/} with respect to event sets $\Sigma_1$, $\Sigma_2$, and $\Sigma_k$, where $\Sigma_1\cap \Sigma_2\subseteq \Sigma_k$, if $K = P_{1+k} (K)\parallel P_{2+k} (K)$, where $P_{i+k}: (\Sigma_1\cup \Sigma_2)^*\to (\Sigma_i\cup \Sigma_k)^*$ is a projection, for $i=1,2$. Note that $\Sigma_k$ can always be extended so that the language $K$ becomes conditionally decomposable. A polynomial algorithm to compute such an extension can be found in~\cite{SCL12}. On the other hand, however, to find the minimal extension (with respect to set inclusion) is NP-hard~\cite{JDEDS}.

  Now we recall the coordination control problem that is discussed in this paper.
  \begin{prob}\label{problem:controlsynthesis}
    Consider two generators $G_1$ and $G_2$ over the event sets $\Sigma_1$ and $\Sigma_2$, respectively, and a generator $G_k$ (called a {\em coordinator\/}) over the event set $\Sigma_k$ satisfying the inclusions $\Sigma_1\cap\Sigma_2\subseteq \Sigma_k\subseteq \Sigma_1\cup\Sigma_2$. Let $K \subseteq L_m(G_1 \parallel G_2 \parallel G_k)$ be a specification language. Assume that $K$ and its prefix-closure $\overline{K}$ are conditionally decomposable with respect to event sets $\Sigma_1$, $\Sigma_2$, and $\Sigma_k$. The aim of coordination control is to determine nonblocking supervisors $S_1$, $S_2$, and $S_k$ for the respective generators such that $L_m(S_k/G_k)\subseteq P_k(K)$, $L_m(S_i/ [G_i \parallel (S_k/G_k) ])\subseteq P_{i+k}(K)$, for $i=1,2$, and $L_m(S_1/ [G_1 \parallel (S_k/G_k) ]) \parallel L_m(S_2/ [G_2 \parallel (S_k/G_k) ]) = K$.
  \end{prob}
  
  One possible way how to construct a coordinator is to set $G_k=P_k(G_1)\parallel P_k(G_2)$, see~\cite{automatica2011,JDEDS} for more details. An advantage of this construction is that the coordinator does not affect the system, that is, $G_1\parallel G_2\parallel G_k = G_1\parallel G_2$.

  The notion of conditional controllability introduced in~\cite{KvS08} and further studied in~\cite{automatica2011,JDEDS,cdc2014} plays the central role in coordination control. In what follows, we use the notation $\Sigma_{i,u}=\Sigma_i\cap \Sigma_u$ to denote the set of uncontrollable events of the event set $\Sigma_i$. 
  
  Let $G_1$ and $G_2$ be generators over the event sets $\Sigma_1$ and $\Sigma_2$, respectively, and let $G_k$ be a coordinator over the event set $\Sigma_k$. Let $P_k:\Sigma^*\to \Sigma_k^*$ and $P_{i+k}:\Sigma^*\to (\Sigma_i\cup\Sigma_k)^*$ be projections. A language $K\subseteq L_m(G_1\parallel G_2\parallel G_k)$ is {\em conditionally controllable\/} with respect to generators $G_1$, $G_2$, $G_k$ and uncontrollable event sets $\Sigma_{1,u}$, $\Sigma_{2,u}$, $\Sigma_{k,u}$ if
  (i) $P_k(K)$ is controllable with respect to $L(G_k)$ and $\Sigma_{k,u}$ and
  (ii) $P_{i+k}(K)$ is controllable with respect to $L(G_i) \parallel \overline{P_k(K)}$ and $\Sigma_{i+k,u}$, for $i=1,2$, where $\Sigma_{i+k,u}=(\Sigma_i\cup \Sigma_k)\cap \Sigma_u$.
  The supremal conditionally controllable sublanguage always exists and equals to the union of all conditionally controllable sublanguages~\cite{JDEDS}. 
  
  Consider the setting of Problem~\ref{problem:controlsynthesis} and define the languages
  \begin{equation}
    \begin{aligned}\label{eqCC}
      \supC_k     & =  \supC(P_k(K), L(G_k), \Sigma_{k,u})\\[3pt]
      \supC_{i+k} & =  \supC(P_{i+k}(K), L(G_i) \parallel \overline{\supC_k}, \Sigma_{i+k,u})
    \end{aligned}
  \end{equation}
  for $i=1,2$, where $\supC(K,L,\Sigma_u)$ denotes the supremal controllable sublanguage of $K$ with respect to $L$ and $\Sigma_u$, see~\cite{CL08}. Let $\supCC(K, L, (\Sigma_{1,u}, \Sigma_{2,u}, \Sigma_{k,u}))$ denote the supremal conditionally controllable sublanguage of $K$ with respect to $L=L(G_1\parallel G_2\parallel G_k)$ and sets of uncontrollable events $\Sigma_{1,u}$, $\Sigma_{2,u}$, $\Sigma_{k,u}$. In~\cite{JDEDS}, we have shown that $P_k(\supC_{i+k})\subseteq \supC_k$ and that if in addition the converse inclusion also holds, then $\supC_{1+k} \parallel \supC_{2+k} = \supCC(K, L, (\Sigma_{1,u}, \Sigma_{2,u}, \Sigma_{k,u}))$. This has been further improved by introducing a weaker condition for nonconflicting supervisors in~\cite{cdc2014}. Recall that two languages $L_1$ and $L_2$ are {\em nonconflicting\/} if $\overline{L_1\parallel L_2} = \overline{L_1} \parallel \overline{L_2}$.
  
  \begin{thm}[\cite{cdc2014}]\label{thmNEW}
    Consider the setting of Problem~\ref{problem:controlsynthesis} and the languages defined in~(\ref{eqCC}). Assume that the languages $\supC_{1+k}$ and $\supC_{2+k}$ are nonconflicting. If $P_k(\supC_{1+k}) \cap P_k(\supC_{2+k})$ is controllable with respect to $L(G_k)$ and $\Sigma_{k,u}$, then $\supC_{1+k} \parallel \supC_{2+k} = \supCC(K, L, (\Sigma_{1,u}, \Sigma_{2,u}, \Sigma_{k,u}))$, where $L=L(G_1\parallel G_2\parallel G_k)$.
    \qed
  \end{thm}

  For coordination control, the notion of conditional observability is of the same importance as observability for supervisory control theory. 
  
  Let $G_1$ and $G_2$ be generators over the event sets $\Sigma_1$ and $\Sigma_2$, respectively, and let $G_k$ be a coordinator over $\Sigma_k$. A language $K\subseteq L_m(G_1\parallel G_2\parallel G_k)$ is {\em conditionally observable\/} with respect to generators $G_1$, $G_2$, $G_k$, controllable sets $\Sigma_{1,c}$, $\Sigma_{2,c}$, $\Sigma_{k,c}$, and projections $Q_{1+k}$, $Q_{2+k}$, $Q_{k}$, where $Q_i: \Sigma_i^*\to \Sigma_{i,o}^*$, for $i=1+k,2+k,k$, if
  (i) $P_k(K)$ is observable with respect to $L(G_k)$, $\Sigma_{k,c}$, and $Q_{k}$, and
  (ii) $P_{i+k}(K)$ is observable with respect to $L(G_i) \parallel \overline{P_k(K)}$, $\Sigma_{i+k,c}$, and $Q_{i+k}$, for $i=1,2$, where $\Sigma_{i+k,c}=\Sigma_c \cap (\Sigma_i \cup \Sigma_k)$.
  
  Analogously to the notion of $L_m(G)$-closed languages, we recall the notion of conditionally-closed languages defined in~\cite{ifacwc2011}. A nonempty language $K$ over the event set $\Sigma$ is {\em conditionally closed\/} with respect to generators $G_1$, $G_2$, $G_k$ if
  (i) $P_k(K)$ is $L_m(G_k)$-closed, and
  (ii) $P_{i+k}(K)$ is $(L_m(G_i) \parallel P_k(K))$-closed, for $i=1,2$.
  
  \begin{thm}[\cite{cdc2014}]\label{thm1}
    Consider the setting of Problem~\ref{problem:controlsynthesis}. There exist nonblocking supervisors $S_1$, $S_2$, $S_k$ as required in Problem~\ref{problem:controlsynthesis} if and only if the specification $K$ is
    (i) conditionally controllable with respect to generators $G_1$, $G_2$, $G_k$ and $\Sigma_{1,u}$, $\Sigma_{2,u}$, $\Sigma_{k,u}$, 
    (ii) conditionally closed with respect to generators $G_1$, $G_2$, $G_k$, and 
    (iii) conditionally observable with respect to $G_1$, $G_2$, $G_k$, event sets $\Sigma_{1,c}$, $\Sigma_{2,c}$, $\Sigma_{k,c}$, and projections $Q_{1+k}$, $Q_{2+k}$, $Q_{k}$ from $\Sigma_i^*$ to $\Sigma_{i,o}^*$, for $i=1+k,2+k,k$.
    \qed
  \end{thm}
  
  Note that for prefix-closed languages, we do not need nonconflictingness and conditional closedness, because they are automatically satisfied for prefix-closed languages.

\section{Conditional Relative Observability}\label{sec:cro}
  As mentioned above, relative observability (with respect to $C$, or just $C$-observability) has been introduced and studied in~\cite{caiCDC13} as a weaker condition than normality, but stronger than observability. It has been shown there that supremal relatively observable sublanguages exist. 
  
  In this section, we introduce the notion of conditional $C$-observability (or conditional relative observability with respect to $C$) in a similar way we have defined conditional observability or conditional normality, as a counterpart of relative observability for coordination control. First, we recall the definition of relative observability.

  Let $K \subseteq C \subseteq L_m(G)$. The language $K$ is {\em $C$-observable\/} with respect to a plant $G$ and a projection $Q:\Sigma^*\to\Sigma_o^*$ (we also say that $K$ is relatively observable with respect to $C$, $G$, and $Q$) if for all words $s, s'\in \Sigma^*$ such that $Q(s) = Q(s')$ it holds that for all $\sigma\in \Sigma$, $s\sigma \in \overline{K}$, $s' \in \overline{C}$, and $s'\sigma \in L(G)$ imply that $s'\sigma \in \overline{K}$. Note that for $C=K$ the definition coincides with the definition of observability.

  \begin{defn}\label{def:conditionalRobservability}
    Let $G_1$ and $G_2$ be generators over the event sets $\Sigma_1$ and $\Sigma_2$, respectively, and let $G_k$ be a coordinator over the event set $\Sigma_k$. Let $K\subseteq C \subseteq L_m(G_1\parallel G_2\parallel G_k)$. The language $K$ is {\em conditionally $C$-observable\/} with respect to generators $G_1, G_2, G_k$, and projections $Q_{1+k}, Q_{2+k}, Q_{k}$, where $Q_i: \Sigma_i^*\to \Sigma_{i,o}^*$, for $i=1+k,2+k,k$ if
    \begin{enumerate}
      \item $P_k(K)$ is $P_k(C)$-observable with respect to $L(G_k)$ and $Q_{k}$, and
      \item $P_{i+k}(K)$ is $P_{i+k}(C)$-observable with respect to $L(G_i) \parallel \overline{P_k(K)}$ and $Q_{i+k}$, for $i=1,2$.
    \end{enumerate}
  \end{defn}

  As relative observability implies observability~\cite{caiCDC13}, we immediately obtain the following result from Theorem~\ref{thm1}.
  \begin{thm}
    Consider the setting of Problem~\ref{problem:controlsynthesis}. Let $K\subseteq C\subseteq L_m(G_1\parallel G_2\parallel G_k)$. If the specification $K$ is conditionally controllable with respect to $G_1, G_2, G_k$ and $\Sigma_{1,u}, \Sigma_{2,u}, \Sigma_{k,u}$, conditionally closed with respect to $G_1, G_2, G_k$, and conditionally $C$-observable with respect generators $G_1, G_2, G_k$ and projections $Q_{1+k}, Q_{2+k}, Q_{k}$ from $\Sigma_i^*$ to $\Sigma_{i,o}^*$, for $i=1+k,2+k,k$, then there exist nonblocking supervisors $S_1$, $S_2$, $S_k$ as required in Problem~\ref{problem:controlsynthesis}.
    \qed
  \end{thm}
  
  In the following example we show that, unlike relative observability, conditional relative observability is not closed under language unions.
  \begin{exmp}\label{ex01}
    Let $L(G_1) = \overline{\{a, \tau a\}}$, $L(G_2) = \overline{\{\tau\}}$, $K_1 = \{a\}$, $K_2 = \{\tau\}$, $\Sigma_k=\{\tau\}$ and $\Sigma_o=\{a\}$. Define $G_k = P_k(G_1) \parallel P_k(G_2)$. It can be verified that both $K_1$ and $K_2$ are conditionally $C$-observable, for $C=K_1\cup K_2$. We now show that $K_1\cup K_2$ is not conditionally $C$-observable. To see this, let $Q_{1+k}:\{a,\tau\}^*\to \{a\}^*$ be the observation projection. Then $Q_{1+k}(\varepsilon)=Q_{1+k}(\tau)$, $\varepsilon a\in P_{1+k}(K_1\cup K_2)=\{a,\tau\}=P_{1+k}(C) \ni\tau$ and $\tau a\in L_1\parallel\overline{P_k(K_1\cup K_2)}=L_1$, but $\tau a \notin P_{1+k}(K_1\cup K_2)$.
  \end{exmp}
  
  To cope with this issue, we now modify the definition to obtain a stronger version that is closed under language unions. The modification is that we do not require $P_{i+k}(K)$ to be $P_{i+k}(C)$-observable with respect to $L(G_i)\parallel \overline{P_k(K)}$, but with respect to a bigger language $L(G_i)\parallel L(G_k)$.
  \begin{defn}\label{def:StrongConditionalRobservability}
    Let $G_1$ and $G_2$ be generators over the event sets $\Sigma_1$ and $\Sigma_2$, respectively, and let $G_k$ be a coordinator over the event set $\Sigma_k$. Let $K\subseteq C \subseteq L_m(G_1\parallel G_2\parallel G_k)$. The language $K$ is {\em conditionally strong $C$-observable\/} with respect to generators $G_1, G_2, G_k$, and projections $Q_{1+k}, Q_{2+k}, Q_{k}$, where $Q_i: \Sigma_i^*\to \Sigma_{i,o}^*$, for $i=1+k,2+k,k$ if
    \begin{enumerate}
      \item $P_k(K)$ is $P_k(C)$-observable with respect to $L(G_k)$ and $Q_{k}$, and
      \item $P_{i+k}(K)$ is $P_{i+k}(C)$-observable with respect to $L(G_i) \parallel L(G_k)$ and $Q_{i+k}$, for $i=1,2$.
    \end{enumerate}
  \end{defn}
  
  Note that, by definition, if $K'\subseteq K$ is conditionally (strong) $C$-observable, it is also conditionally (strong) $K$-observable.

  We can now prove that the supremal conditionally strong relative observable sublanguage always exists.
  \begin{thm}\label{existenceCRO}
    For a given $C$, the supremal conditionally strong $C$-observable sublanguage always exists and equals to the union of all conditionally strong $C$-observable sublanguages.
  \end{thm}
  \begin{proof}
    Let $I$ be an index set, and for $i\in I$, let $K_i\subseteq C$ be a conditionally strong $C$-observable sublanguage of $K\subseteq L_m(G_1\parallel G_2\parallel G_k)$ with respect to generators $G_1$, $G_2$, $G_k$ and projections $Q_{1+k}$, $Q_{2+k}$, $Q_{k}$. We prove that $\cup_{i\in I} K_i$ is conditionally strong $C$-observable.

    To prove that $P_k(\cup_{i\in I} K_i)$ is $P_k(C)$-observable with respect to $L(G_k)$ and $Q_{k}$, let $sa\in P_k(\overline{\cup_{i\in I} K_i}) = \cup_{i\in I} P_k(\overline{K_i})$, $s'\in \overline{P_k(C)}$, $s'a \in L(G_k)$, and $Q_k(s)=Q_k(s')$. Then $sa\in P_k(\overline{K_i})$, for some $i\in I$, and $P_k(C)$-observability of $P_k(K_i)$ with respect to $L(G_k)$ and $Q_k$ implies that $s'a \in P_k(\overline{K_i})\subseteq P_k(\cup_{i\in I} \overline{K_i}) = P_k(\overline{\cup_{i\in I} K_i})$.
     
    To prove that $P_{1+k}(\cup_{i\in I} K_i)$ is $P_{1+k}(C)$-observable, assume that $sa \in P_{1+k}(\overline{\cup_{i\in I} K_i}) = \cup_{i\in I} P_{1+k}(\overline{K_i})$, $s'\in P_{1+k}(\overline{C})$, $s'a\in L(G_1)\parallel L(G_k)$, and $Q_{1+k}(s)=Q_{1+k}(s')$. Then we have that $sa\in P_{1+k}(\overline{K_i})$, for some $i\in I$, and $P_{1+k}(C)$-observability of $P_{1+k}(K_i)$ with respect to $L(G_1)\parallel L(G_k)$ and $Q_{1+k}$ implies that $s'a \in P_{1+k}(\overline{K_i})$. 
    
    The case for $P_{2+k}(\overline{\cup_{i\in I} K_i})$ is $P_{2+k}(C)$-observable is analogous.
  \end{proof}
  
  We now recall definitions of normality and conditional normality, and compare the notion of conditional normality to conditional (strong) relative observability.
  
  Let $G$ be a generator over the event set $\Sigma$, and let $Q:\Sigma^* \to \Sigma_o^*$ be a projection. A language $K\subseteq L_m(G)$ is {\em normal\/} with respect to $L(G)$ and $Q$ if $\overline{K} = Q^{-1}Q(\overline{K})\cap L(G)$. It is known that normality implies observability~\cite{CL08}. 
  
  Let $G_1$ and $G_2$ be generators over the event sets $\Sigma_1$ and $\Sigma_2$, respectively, and let $G_k$ be a coordinator over $\Sigma_k$. A language $K\subseteq L_m(G_1\parallel G_2\parallel G_k)$ is {\em conditionally normal\/} with respect to generators $G_1, G_2, G_k$ and projections $Q_{1+k}, Q_{2+k}$, $Q_{k}$, where $Q_i: \Sigma_i^*\to \Sigma_{i,o}^*$, for $i=1+k,2+k,k$, if
  (i) $P_k(K)$ is normal with respect to $L(G_k)$ and $Q_{k}$, and
  (ii) $P_{i+k}(K)$ is normal with respect to $L(G_i) \parallel \overline{P_k(K)}$ and $Q_{i+k}$, for $i=1,2$, cf.~\cite{cdc2014}.

  The following theorem compares the notions of conditional observability, conditional normality, conditional relative observability, and conditional strong relative observability. The main point of this result is to show that we do not need to use conditional normality in coordination control anymore, because the weaker condition of conditional strong relative observability can be used instead.
  \begin{thm}\label{thmImplications}
    The following holds:
    \begin{enumerate}
      \item Conditional normality implies conditional strong relative observability.
      \item Conditional strong relative observability implies conditional relative observability.
      \item Conditional relative observability implies conditional observability.
    \end{enumerate}
  \end{thm}
  \begin{proof}
    The implication (2) is obvious by definition, because $\overline{P_k(K)}\subseteq L(G_k)$, while (3) follows from~\cite{caiCDC13} where it was shown that relative observability implies observability. We now prove (1). Let $K\subseteq C \subseteq L_m(G_1\parallel G_2\parallel G_k)$ be such that $K$ is conditionally normal with respect to generators $G_1,G_2,G_k$ and projections $Q_{1+k}, Q_{2+k}, Q_k$. Then, the assumption that $P_k(K)$ is normal with respect to $L(G_k)$ implies that $P_k(K)$ is $P_k(C)$-observable with respect to $L(G_k)$ by~\cite{caiCDC13}. Moreover, for $i=1,2$, we have that $P_{i+k}(K)$ is normal with respect to $L(G_i)\parallel\overline{P_k(K)}$. By Lemma~\ref{normalitaComposition}, $L(G_i)\parallel\overline{P_k(K)}$ is normal with respect to $L(G_i)\parallel L(G_k)$. Hence, by the transitivity of normality (Lemma~\ref{lem_trans}), $P_{i+k}(K)$ is normal with respect to $L(G_i)\parallel L(G_k)$. Then, by~\cite{caiCDC13}, we obtain that $P_{i+k}(K)$ is $P_{i+k}(C)$-observable with respect to $L(G_i)\parallel L(G_k)$, which was to be shown.
  \end{proof}

  Note that the language $K_1$ from Example~\ref{ex01} is conditionally relative observable, but not conditionally strong relative observable (and therefore not conditionally normal). On the other hand, $K_2$ is conditionally normal, hence also conditionally (strong) relative observable. Note also that conditional strong relative observability does not imply conditional normality, see, e.g., condition (i) of the definitions.

  We have shown that the supremal conditionally controllable and conditionally strong relative observable sublanguage exists. We now present conditions under which a conditionally controllable and conditionally observable sublanguage containing the supremal conditionally controllable and conditionally strong relative observable sublanguage can be computed in a distributed way.
  
  Consider the setting of Problem~\ref{problem:controlsynthesis} and define the languages 
  \begin{equation}\label{eqCRO}
    \begin{aligned}
      \supCRO_k     & = \supCRO(P_k(K), L(G_k))\\[3pt]
      \supCRO_{i+k} & = \supCRO(P_{i+k}(K), L(G_i) \parallel \overline{\supCRO_k})
    \end{aligned}
  \end{equation}
  for $i=1,2$, where $\supCRO(K,L)$ denotes the supremal controllable (with respect to the corresponding event set of uncontrollable events) and $(K\cap L)$-observable (with respect to corresponding projection to observable events) sublanguage of the language $K$. The way how to compute the supremal relatively observable sublanguage is described in~\cite{caiCDC13}. For $K\subseteq L$, let
  \[
    \supcCSRO(K,L,(\Sigma_{1,u},\Sigma_{2,u},\Sigma_{k,u}),(Q_{1+k},Q_{2+k},Q_{k}))
  \]
  denote the supremal conditionally controllable and conditionally strong $K$-observable sublanguage of the specification language $K$ with respect to the plant language $L=L(G_1\parallel G_2\parallel G_k)$, the sets of uncontrollable events $\Sigma_{1,u}$, $\Sigma_{2,u}$, $\Sigma_{k,u}$, and projections $Q_{1+k}$, $Q_{2+k}$, $Q_{k}$, where $Q_i:\Sigma_i^* \to \Sigma^*_{i,o}$, for $i=1+k,2+k,k$. For simplicity, denote $\supcCSRO = \supcCSRO(K, L, (\Sigma_{1,u}, \Sigma_{2,u}, \Sigma_{k,u}), (Q_{1+k}, Q_{2+k}, Q_{k}))$. It can be shown that 
  \begin{equation}\label{eqInc}
    \supcCSRO \subseteq \supCRO_{1+k} \parallel \supCRO_{2+k}\,.
  \end{equation}
  By Lemma~\ref{lem11} we need to show that $P_{i+k}(\supcCSRO)\subseteq \supCRO_{i+k}$, for $i=1,2$. By definition of conditional controllability, $P_{i+k}(\supcCSRO) \subseteq P_{i+k}(K)$ is controllable with respect to $L(G_i) \parallel \overline{P_k(\supcCSRO)}$. Since $P_k(\supcCSRO)\subseteq P_k(K)$ is controllable and $P_k(K)$-observable with respect to $L(G_k)$, $P_k(\supcCSRO)\subseteq \overline{\supCRO_k}$. Thus, $P_k(\supcCSRO)$ is controllable with respect to $\overline{\supCRO_k}\subseteq L(G_k)$. Then, by Lemma~\ref{feng}, $L(G_i)\parallel \overline{P_k(\supcCSRO)}$ is controllable with respect to $L(G_i)\parallel \overline{\supCRO_k}$, and the transitivity of controllability (Lemma~\ref{lem_transC}) implies that $P_{i+k}(\supcCSRO)$ is controllable with respect to $L(G_i) \parallel \overline{\supCRO_k}$. 
  Next, by definition of conditional strong relative observability, $P_{i+k}(\supcCSRO)$ is $P_{i+k}(K)$-observable with respect to $L(G_i)\parallel L(G_k)$, hence it is also $C$-observable with respect to $L(G_i)\parallel L(G_k)$, for every $P_{i+k}(\supcCSRO) \subseteq C \subseteq P_{i+k}(K)$. As $P_{i+k}(\supcCSRO) \subseteq L(G_i)\parallel \overline{\supCRO_k}$, we also obtain that $P_{i+k}(\supcCSRO)$ is $C'$-observable with respect to $L(G_i)\parallel \overline{\supCRO_k}$, for every $P_{i+k}(\supcCSRO) \subseteq C' \subseteq P_{i+k}(K) \cap (L(G_i)\parallel \overline{\supCRO_k})$, which means that $P_{i+k}(\supcCSRO)\subseteq \supCRO_{i+k}$.
  
  This says that if $\supCRO_{1+k} \parallel \supCRO_{2+k}$ is conditionally controllable and conditionally observable, we have computed a language that is at least as good a solution as the supremal conditionally controllable and conditionally strong $K$-observable sublanguage, which is now the weakest known condition for which the supremal sublanguage exists.

  We now formulate the main result. 
 \begin{thm}\label{thm25}
    Consider the setting of Problem~\ref{problem:controlsynthesis} and the languages defined in~(\ref{eqCRO}). Assume that $\supCRO_{1+k}$ and $\supCRO_{2+k}$ are nonconflicting, and let us denote $M = \supCRO_{1+k} \parallel \supCRO_{2+k}$ and $L=L(G_1\parallel G_2\parallel G_k)$. If $P_k(M)$ is controllable and $P_k(C)$-observable with respect to $L(G_k)$, $\Sigma_{k,u}$, and $Q_k$, for some $M\subseteq C \subseteq L$, then $M$ is conditionally controllable with respect to $G_1$, $G_2$, $G_k$ and $\Sigma_{1,u}$, $\Sigma_{2,u}$, $\Sigma_{k,u}$, and conditionally observable with respect to $G_1$, $G_2$, $G_k$ and $Q_{1+k}$, $Q_{2+k}$, $Q_{k}$. Moreover, it contains the language $\supcCSRO$.
  \end{thm}
  \begin{proof}
     Indeed, $M \subseteq P_{1+k}(K)\parallel P_{2+k}(K) = K$ by conditional decomposability, and $P_k(M)$ is controllable and $P_k(M)$-observable with respect to $L(G_k)$, $\Sigma_{k,u}$, $Q_k$ by assumptions (since $P_k(C)$-observability implies $P_k(C')$-observability for every $M\subseteq C'\subseteq C$). Next, $P_{1+k}(M) = \supCRO_{1+k}\parallel P_k(M)$ is controllable with respect to $[L(G_1)\parallel \overline{\supCRO_k}]\parallel \overline{P_k(M)} = L(G_1)\parallel \overline{P_k(M)}$ by Lemma~\ref{feng} (because the nonconflictingness of $\supCRO_{1+k}$ and $\supCRO_{2+k}$ implies the nonconflictingness of $\supCRO_{1+k}$ and $P_k(M)$) and Lemma~\ref{lemmaZ}. 
    To show that $P_{1+k}(M)\subseteq P_{1+k}(K)\cap (L(G_1) \parallel \overline{\supCRO_k})$ is $P_{1+k}(M)$-observable, let $a\in \Sigma_{1+k}$, $sa,s'\in \overline{P_{1+k}(M)}$, $s'a\in L(G_1)\parallel \overline{P_k(M)}\subseteq L(G_1) \parallel \overline{\supCRO_k}$, and $Q_{1+k}(s)=Q_{1+k}(s')$. By the $(P_{1+k}(K)\cap (L(G_1) \parallel \overline{\supCRO_k}))$-observability of $\supCRO_{1+k}$, $s'a\in \overline{\supCRO_{1+k}}$. 
    We have two cases: 
    (i) If $a\in\Sigma_1\setminus\Sigma_k$, then $P_k(s'a)=P_k(s')\in \overline{P_k(M)}\subseteq \overline{P_k(\supCRO_{2+k})}$. 
    (ii) If $a\in\Sigma_k$, then $P_k(s)a\in \overline{P_k(M)}$, $P_k(s')\in \overline{P_k(M)}$, and $P_k(s')a\in L(G_k)$ imply (by $P_k(M)$-observability of $P_k(M)$) that $P_k(s'a)\in \overline{P_k(M)}\subseteq \overline{P_k(\supCRO_{2+k})}$. 
    Therefore, in both cases, $s'a \in \overline{\supCRO_{1+k}} \parallel \overline{P_k(\supCRO_{2+k})} = \overline{P_{1+k}(M)}$ by the nonconflictingness. The case of $P_{2+k}(M)$ is analogous, hence $M$ is conditionally controllable with respect to $G_1$, $G_2$, $G_k$ and $\Sigma_{1,u}$, $\Sigma_{2,u}$, $\Sigma_{k,u}$, and conditionally $M$-observable (hence observable) with respect to $G_1$, $G_2$, $G_k$ and $Q_{1+k}$, $Q_{2+k}$, $Q_{k}$.
    Finally, $\supcCSRO\subseteq \supCRO_{1+k}\parallel\supCRO_{2+k}$ as shown in (\ref{eqInc}) above.
  \end{proof}

\section{Auxiliary Results}
  This section provides auxiliary results needed in the paper.
  \begin{lem}\label{lem_trans}
    Let $K\subseteq L\subseteq M$ be languages such that $K$ is normal with respect to $L$ and $Q$, and $L$ is normal with respect to $M$ and $Q$. Then $K$ is normal with respect to $M$ and $Q$.
  \end{lem}
  \begin{proof}
    By the assumption $Q^{-1}Q(\overline{K})\cap \overline{L} = \overline{K}$ and $Q^{-1}Q(\overline{L})\cap \overline{M} = \overline{L}$, hence $Q^{-1}Q(\overline{K})\cap \overline{M} \subseteq Q^{-1}Q(\overline{L})\cap \overline{M} = \overline{L}$. This implies that $Q^{-1}Q(\overline{K})\cap \overline{M} = Q^{-1}Q(\overline{K})\cap \overline{M} \cap \overline{L} = \overline{K} \cap \overline{M} = \overline{K}$.
  \end{proof}

  \begin{lem}\label{normalitaComposition}
    Let $K_1\subseteq L_1$ over $\Sigma_1$ and $K_2\subseteq L_2$ over $\Sigma_2$ be nonconflicting languages such that $K_1$ is normal with respect to $L_1$ and $Q_1:\Sigma_1^*\to \Sigma_{1,o}^*$ and $K_2$ is normal with respect to $L_2$ and $Q_2:\Sigma_2^*\to \Sigma_{2,o}^*$, where $L_1$ and $L_2$ are prefix-closed. Then $K_1\parallel K_2$ is normal with respect to $L_1\parallel L_2$ and $Q:(\Sigma_1\cup\Sigma_2)^*\to (\Sigma_{1,o}\cup\Sigma_{2,o})^*$.
  \end{lem}
  \begin{proof}
    By definition we have that $Q^{-1}Q(\overline{K_1\parallel K_2}) \cap L_1\parallel L_2 \subseteq Q_1^{-1}Q_1(\overline{K_1}) \parallel Q_2^{-1}Q_2(\overline{K_2}) \parallel L_1\parallel L_2 = \overline{K_1} \parallel \overline{K_2} = \overline{K_1\parallel K_2}$, where the first equality is by normality of $K_1$ and $K_2$, and the last equality is by nonconflictingness. As the other inclusion always holds, the proof is complete.
  \end{proof}
  
  \begin{lem}[Proposition~4.6 in \cite{FLT}]\label{feng}
    Let $L_i\subseteq \Sigma_i^*$, for $i=1,2$, be prefix-closed languages, and let $K_i\subseteq L_i$ be controllable with respect to $L_i$ and $\Sigma_{i,u}$. Let $\Sigma=\Sigma_1\cup \Sigma_2$. If $K_1$ and $K_2$ are nonconflicting, then $K_1\parallel K_2$ is controllable with respect to $L_1\parallel L_2$ and $\Sigma_u$. 
    \qed
  \end{lem}

  \begin{lem}[\cite{automatica2011}]\label{lem_transC}
    Let $K\subseteq L \subseteq M$ be languages over $\Sigma$ such that $K$ is controllable with respect to $\overline{L}$ and $\Sigma_u$, and $L$ is controllable with respect to $\overline{M}$ and $\Sigma_u$. Then $K$ is controllable with respect to $\overline{M}$ and $\Sigma_u$. 
    \qed
  \end{lem}

  \begin{lem}[\cite{automatica2011}]\label{lem11}
    Let $L_i \subseteq \Sigma_i^*$, for $i=1,2$, and let $P_i : (\Sigma_1\cup \Sigma_2)^* \to \Sigma_i^*$ be a projection. Let $A\subseteq (\Sigma_1\cup \Sigma_2)^*$ such that $P_1(A)\subseteq L_1$ and $P_2(A)\subseteq L_2$. Then $A \subseteq L_1\parallel L_2$. 
    \qed
  \end{lem}

  \begin{lem}\label{lemmaZ}
    Consider the setting of Problem~\ref{problem:controlsynthesis}, and the languages defined in (\ref{eqCRO}). Then $P_k(\supCRO_{i+k}) \subseteq \supCRO_k$, for $i=1,2$.
  \end{lem}
  \begin{proof}
    By definition, $P_k(\supCRO_{i+k})\subseteq \overline{\supCRO_k}\cap P_k(K)$. We prove $\overline{\supCRO_k}\cap P_k(K)\subseteq \supCRO_k$ by showing that $\overline{\supCRO_k}\cap P_k(K)$ is controllable with respect to $L(G_k)$ and $C_k$-observable with respect to $L(G_k)$, for some fixed $C_k$.
    Let $s \in \overline{\overline{\supCRO_k}\cap P_k(K)}$, $u \in \Sigma_{k,u}$, and $su \in L(G_k)$. By controllability of $\supCRO_k$, $su \in \overline{\supCRO_k}\subseteq \overline{P_k(K)}$, hence there exists $v$ such that $suv \in \supCRO_k\subseteq P_k(K)$. Hence, $suv \in \overline{\supCRO_k}\cap P_k(K)$, and $su\in \overline{\overline{\supCRO_k}\cap P_k(K)}$.
    Let $s, s'\in \Sigma^*$ and $\sigma\in \Sigma$ be such that $Q_k(s) = Q_k(s')$, $s\sigma \in \overline{\overline{\supCRO_k}\cap P_k(K)}$, $s' \in \overline{C_k}$, and $s'\sigma \in L(G_k)$. By $C_k$-observability of $\supCRO_k$, $s'\sigma \in \overline{\supCRO_k}$, and similarly as above we show that $s'\sigma \in \overline{\overline{\supCRO_k}\cap P_k(K)}$. 
  \end{proof}

\section{Conclusion}\label{conclusion}
  In this paper, we have introduced and studied the notion of conditional relative observability, and a coordinated computation of a conditionally controllable and conditionally observable sublanguage that contains the supremal conditionally controllable and conditionally strong relative observable sublanguage of the specification language. It is worth mentioning that there exist conditions, namely the observer and OCC (or LCC) properties, that can be fulfilled by a modification of the coordinator event set, and that imply that the assumptions for controllability of Theorem~\ref{thm25} are satisfied. On the other hand, however, to the best of our knowledge, there are no known conditions that could be fulfilled by a simple action on the event sets of the coordinator, so that it would make the conditions for relative observability of Theorem~\ref{thm25} satisfied. This is an interesting topic for the future investigation.

\begin{ack}
  This research was supported by the M\v{S}MT grant LH13012 (MUSIC) and by RVO: 67985840.
\end{ack}

\bibliographystyle{plain}
\bibliography{cdc2014}

\begin{thebibliography}{10}

\bibitem{caiCDC13}
K.~Cai, R.~Zhang, and W.~M. Wonham.
\newblock On relative observability of discrete-event systems.
\newblock In {\em Proc. of CDC 2013}, pages 7285--7290, Florence, Italy, 2013.

\bibitem{CL08}
C.~G. Cassandras and S.~Lafortune.
\newblock {\em Introduction to discrete event systems, Second edition}.
\newblock Springer, 2008.

\bibitem{FLT}
L.~Feng.
\newblock {\em Computationally Efficient Supervisor Design for Discrete-Event
  Systems}.
\newblock PhD thesis, University of Toronto, 2007.

\bibitem{ifacwc2011}
J.~Komenda, T.~Masopust, and J.~H. van Schuppen.
\newblock Coordinated control of discrete event systems with non\-prefix-closed
  languages.
\newblock In {\em Proc. of IFAC World Congress 2011}, pages 6982--6987, Milano,
  Italy, 2011.

\bibitem{SCL12}
J.~Komenda, T.~Masopust, and J.~H. van Schuppen.
\newblock On conditional decomposability.
\newblock {\em Systems Control Lett.}, 61(12):1260--1268, 2012.

\bibitem{automatica2011}
J.~Komenda, T.~Masopust, and J.~H. van Schuppen.
\newblock Supervisory control synthesis of discrete-event systems using a
  coordination scheme.
\newblock {\em Automatica}, 48(2):247--254, 2012.

\bibitem{JDEDS}
J.~Komenda, T.~Masopust, and J.~H. van Schuppen.
\newblock Coordination control of discrete-event systems revisited.
\newblock {\em Discrete Event Dyn. Syst.}, 2014.
\newblock to appear, DOI: 10.1007/s10626-013-0179-x.

\bibitem{cdc2014}
J.~Komenda, T.~Masopust, and J.~H. van Schuppen.
\newblock Maximally permissive coordination supervisory control -- towards
  necessary and sufficient conditions.
\newblock Submitted manuscript. [Online]. Available at {\tt
  http://arxiv.org/abs/1403.4762}, 2014.

\bibitem{KvS08}
J.~Komenda and J.~H. van Schuppen.
\newblock Coordination control of discrete event systems.
\newblock In {\em Proc. of WODES 2008}, pages 9--15, Gothenburg, Sweden, 2008.

\bibitem{RW89}
P.~J. Ramadge and W.~M. Wonham.
\newblock The control of discrete event systems.
\newblock {\em Proc. of IEEE}, 77(1):81--98, 1989.

\end{thebibliography}

\end{document}